\documentclass[10pt]{article}%
\usepackage{amsmath}
\usepackage{amsfonts}
\usepackage{mathrsfs}
\usepackage[utf8]{inputenc}
\usepackage{amssymb, color}
\usepackage[linkcolor=black,anchorcolor=black,citecolor=black]{hyperref}
\usepackage{graphicx}
\numberwithin{equation}{section}
\usepackage[body={15.5cm,21cm}, top=3cm]{geometry}%
\providecommand{\U}[1]{\protect\rule{.1in}{.1in}}
\providecommand{\U}[1]{\protect \rule{.1in}{.1in}}
\newtheorem{theorem}{Theorem}[section]

\newtheorem{corollary}[theorem]{Corollary}

\newtheorem{lemma}[theorem]{Lemma}

\newtheorem{proposition}[theorem]{Proposition}
\newtheorem{remark}[theorem]{Remark}

\newtheorem{assumption}[theorem]{Assumption}
\newenvironment{proof}[1][Proof]{\noindent \textbf{#1.} }{\  \rule{0.5em}{0.5em}}

\DeclareMathOperator*{\supp}{supp}

\begin{document}
	\title{Optimal Consumption for Recursive Preferences with Local Substitution -- the Case of Certainty}
	\author{ Hanwu Li\thanks{Research Center for Mathematics and Interdisciplinary Sciences, Shandong University, lihanwu@sdu.edu.cn.}
	\thanks{Funded by the Deutsche Forschungsgemeinschaft (DFG, German Research Foundation) – SFB 1283/2 2021 – 317210226 and  the Frontiers Science Center for Nonlinear Expectations (Ministry of Education), Shandong University, Qingdao 266237, Shandong.}
\and Frank Riedel \thanks{Center for Mathematical Economics, Bielefeld University,		frank.riedel@uni-bielefeld.de.}
\and Shuzhen Yang \thanks{Zhongtai Securities Institute for Financial Studies, Shandong University, yangsz@sdu.edu.cn.}}

			\date{}
		\maketitle
		
		\begin{abstract}
			We characterize optimal consumption policies in a recursive intertemporal utility framework with local  substitution.  We establish  existence and uniqueness and  a version of the Kuhn-Tucker theorem characterizing the optimal consumption plan. An explicit solution is provided for the case when the felicity function is of the Epstein-Zin's type.
			\vskip1em
			


		\end{abstract}

	\section{Introduction}

Intertemporal choices  are important and ubiquitous. They form the basis of most dynamic models in economics and finance, and thus the shape of intertemporal preferences is crucial in understanding a host of economic situations ranging from microeconomic choices involving saving,  health, the environment, or the growth of nations.

	The standard workhorse 	for economic and financial intertemporal decisions has long been the time-additive discounted utility model introduced by Paul Samuelson in 1937 (see \cite{PS37}).  The time-additive discounted utility model is restrictive in many senses. It assumes, among other things, that the marginal rate of substitution between consumption today and consumption in a week does not depend on the consumption in between, to give an example. Time-additive utility also implies that the current utility does not depend on future utility as it is  assumed that the overall utility is a discounted sum of period utilities that depend only on current consumption. We refer to \cite{FLO} for a critical discussion   of the time-additive model's weaknesses.

To overcome these restrictions, recursive models of utility have been introduced and widely studied.
In  discrete-time, recursive utility was developed by Kreps and Porteus \cite{KP}, Epstein and Zin \cite{EZ}, Weil \cite{Weil}, and others, making it possible to disentangle risk aversion from the elasticity of intertemporal substitution. Stochastic differential utility was introduced by Epstein \cite{E87} in a deterministic setting and by Duffie and Epstein \cite{DE} in a stochastic setting as a continuous-time version of recursive utility. Epstein \cite{E87}, Duffie and Epstein \cite{DE} and the subsequent literature define stochastic differential utility axiomatically in continuous time, but do not establish a rigorous connection to discrete-time recursive utility. Heuristic links to recursive utility are provided in Duffie and Epstein \cite{DE}, Svensson \cite{Svensson} and Skiadas \cite{Skiadas}. The optimal consumption and investment problems under recursive utility have been widely studied. We may refer to Schroder and Skiadas \cite{SS99}, \cite{SS03}, Kraft, Seifried and Steffensen \cite{KSS13}, Seiferling and Serfried \cite{KSS17} and Matoussi and Xing \cite{MX}, among others.

These models are still based on the current rate of consumption as the basic ingredient of preferences.  Hindy, Huang and Kreps \cite{HHK} argue convincingly that models based on the rate of consumption do not capture desirable properties of local substitution. In particular, such preferences are not continuous with respect to economically reasonable topologies, as the weak topology. They thus propose to model the basic ingredient of  intertemporal consumption  by  a level of satisfaction, a weighted average of past consumption. Bank and Riedel \cite{BR00} provided an analog of the Kuhn-Tucker theorem which yields the first-order conditions for optimality. Based on this theorem, they obtained the solution in closed form for the case of finite time horizon. When there are stocks can be traded in the market (or in a stochastic setting), the optimal investment-consumption problem for agent whose preference is given by the Hindy-Huang-Kreps type has been studied by \cite{Alvarez,BR,BKR1,BKR2,BKR3,FLR}.

In this paper, we combine  the two approaches by considering a recursive intertemporal utility based on a level of satisfaction. We establish the basic properties of the utility functional, including continuity in the weak topology. We show that the associated intertemporal choice problem has a unique solution.
Generalizing the approach of \cite{BR00}, we establish necessary and sufficient Kuhn-Tucker-type conditions for optimality. Using this theorem, the construction of optimal consumption plan is given. More precisely, under some additional assumptions on the felicity function (see Assumption \ref{a5.1} (ii)), the optimal consumption starts at some time $t_0$ and ends at some time $t_1$. Besides, $t_0$, $t_1$ and optimal consumption $C^*$ satisfy a fully-coupled forward-backward system. Although the existence of solutions to the general forward-backward system is not obtained, we  solve this problem explicitly when the felicity function is of the Epstein-Zin's type, which coincides with the solution obtained in \cite{BR00} as predicted in \cite{SS99}. 

The paper is organized as follows. In Section 2, we formulate the utility maximization problem in details and provide the existence and uniqueness result. In Section 3, we demonstrate the Kuhn-Tucker-like necessary and sufficient conditions for optimality, which leads to the construction for optimal consumption plan in Section 4. Finally, when the felicity function is given by the Epstein-Zin's type, we solve the utility maximization problem explicitly.

\section{Problem formulation and existence and uniqueness result}

Consider an agent living from time $0$ to some time $T>0$ with initial wealth $w\geq 0$. The agent chooses a cumulative consumption plan {$C$, a finite $\sigma$-additive measure on the time interval $[0,T]$ in the set of distribution functions}
\begin{displaymath}
\mathcal{X}:=\{C:[0,T]\rightarrow \mathbb{R}_+: C \, \textrm{nondecreasing and right-continuous}\}.
\end{displaymath}
We assume that the consumption good is traded at some  continuous and strictly positive price $\psi: [0,T] \to \mathbb R_{++}$. The price functional $\Psi(C)$ is defined by
\begin{displaymath}
	\Psi(C):=\int_0^T \psi_t dC_t.
\end{displaymath}
Here and in the remainder of the paper, the integral with respect to $dC$ is over the closed interval. A positive value $C_0>0$ at time $0$ indicates an initial consumption gulp (we define $C_{0-}:=0$), corresponding to a point mass at zero. The budget  set is given by
\begin{displaymath}
	\mathcal{A}(w):=\{C\in\mathcal{X} \textrm{ such that }\Psi(C)\leq w\}
\end{displaymath}
for a strictly positive initial wealth $w>0$.
 Following \cite{HHK}, the current level of satisfaction $Y^C_t$  is given by
\begin{displaymath}
Y_t^C=\eta_t+\int_0^t \theta_{t,s}dC_s,
\end{displaymath}
where $\eta:[0,T]\rightarrow\mathbb{R}_+$ and $\theta:[0,T]^2\rightarrow\mathbb{R}_+$ are continuous functions. The quantity $\theta_{t,s}$ can be seen as the weight assigned at time $t$ to consumption made at time $s\leq t$ and $\eta_t$ describes an exogenously level of satisfaction for time $t$. The agent's recursive utility type is of the form
\begin{equation}\label{3}
U_t^C=\int_t^T f(s,Y_s^C,U_s^C)ds,
\end{equation}
where $f:[0,T]\times\mathbb{R}_+\times\mathbb{R}\rightarrow\mathbb{R}$ is the {intertemporal aggregator}. The  agent   maximizes the utility over all budget feasible consumptions, i.e., we consider the problem
\begin{equation}\label{1}
v(w)=\sup_{C\in\mathcal{A}(w)}U_0^C.
\end{equation}

We make the following standard assumptions on the {intertemporal aggregator} function $f$.
\begin{assumption}\label{a1}
	\begin{description}
		\item[(i)] For each $s\in[0,T]$, $f(s,\cdot,\cdot)$ is strictly concave and continuously differentiable and for any $s\in[0,T]$,  $u\in\mathbb{R}$, $f(s,\cdot,u)$ is strictly increasing;
		\item[(ii)] For any $(s,y)\in[0,T]\times\mathbb{R}_+$, $u,u'\in\mathbb{R}$, there exists a constant $K>0$ such that
		\[|f(s,y,u)-f(s,y,u')|\leq K|u-u'|;\]
		\item[(iii)] For any $s\in[0,T]$, there exist two constants $\alpha\in(0,1)$ and $K>0$ such that
		\[|f(s,y,0)|\leq K(1+|y|^\alpha).\]
	\end{description}
\end{assumption}


The assumption ensures that the recursive utility functional is well-defined.
We also present some basic properties of the recursive utility obtained by Equation \eqref{3} in the next proposition.

\begin{proposition}\label{p1}
	If the intertemporal aggregator function $f$ satisfies Assumption \ref{a1}, then { for each $C\in \mathcal{A}(w)$, there exists a unique solution $U^C$ to Equation \eqref{3}. Besides, the recursive utility  has the following properties.}
	\begin{description}
		\item[(1)]  We have
		\begin{equation}\label{utility}
			|U^C_s-U^{C'}_s|\leq \int_s^T e^{Kt}|f(t,Y_t^C,U_t^C)-f(t,Y_t^{C'},U_t^C)|dt, \ s\in[0,T]
		\end{equation}
		for some constant $K>0$. The utility functional  $U:\mathcal{X}\rightarrow \mathbb{R}$
		is continuous.
		\item[(2)] The utility functional is strictly concave.
	\end{description}
\end{proposition}

\begin{proof}[Proof of Proposition \ref{p1}]
By Assumption \ref{a1}, for fixed $C \in \mathcal{X}$, the ODE \eqref{3} admits a unique solution $U^C$  because the right side is continuous in time and uniformly Lipschitz-continuous in $U^C$.

For property (1),
for any $C,C'\in \mathcal{A}(w)$, set $\widetilde{U}_t=U^C_t-U^{C'}_t$. {By Assumption \ref{a1} (ii), for any $0\leq t\leq T$, we have}
	\begin{align*}
		|\widetilde{U}_t|\leq &\left|\int_t^T (f(s,Y^C_s,U_s^C)-f(s,Y^{C'}_s,U^C_s))ds\right|+ \left|\int_t^T (f(s,Y^{C'}_s,U_s^C)-f(s,Y^{C'}_s,U^{C'}_s))ds\right|\\
		\leq  &\int_t^T |(f(s,Y^C_s,U_s^C)-f(s,Y^{C'}_s,U^C_s)|ds+\int_t^T K|\widetilde{U}_s|ds.
	\end{align*}
By the Gronwall inequality in backward form\footnote{Recall the Gronwall inequality in backward form (see \cite{DE}). Suppose that for any $t\in[0,T]$,
$
h_t\leq \int_t^T (g_s+\alpha h_s)ds+h_T,
$
where $\alpha$ is a constant and $g$ is a given integrable function. Then we have
$
h_t\leq e^{\alpha(T-t)}h_T+\int_t^T e^{\alpha(s-t)}g_s ds.
$.}, we obtain \eqref{utility}.

We next turn to continuity.
We need to prove that if the sequence $\{C^n\}_{n\in\mathbb{N}}\subset \mathcal{A}(w)$ converges weakly to some $C\in \mathcal{A}(w)$, then for any $t\in[0,T]$, we have
	\begin{equation}\label{converge}
			\lim_{n\rightarrow \infty} U^{C^n}_t=U^C_t.
	\end{equation}
 Since the topology of weak convergence of measures on $([0,T],\mathcal{B}([0,T]))$ is metrizable, it suffices to show Equation \eqref{converge}   for a weakly convergent sequence $\{C^n\}_{n\in\mathbb{N}}\subset\mathcal{A}(w)$ that  converges monotonically   to $C\in\mathcal{A}(w)$.

For each fixed $t\in[0,T]$ with $\Delta C_t:= C_t-C_{t-}=0$, the function $s\rightarrow \theta_{t,s} I_{[0,t]}(s)$ is continuous in $dC$-a.e. $s\in[0,T]$. Applying the Portemanteau theorem, we have $Y^{C^n}_t\rightarrow Y^C_t$ for each $t$ with $\Delta C_t:= C_t-C_{t-}=0$. This convergence holds true for $dt$-a.e. Besides, by weak convergence we have $C^n_T\rightarrow C_T$, which yields the uniform boundedness of $C^n_t$  ($t\in[0,T]$, $n\in\mathbb{N}$). Since $\eta$ and $\theta$ are continuous functions, there exists a constant $M>0$, such that for any $t\in[0,T]$
\begin{equation}\label{Y1}
	Y_t^C\leq M(1+C_T). 
\end{equation}
By Assumption \ref{a1} (ii) and (iii), it is easy to check that
\begin{align*}
	|U^C_t|=&|\int_t^T f(s,Y^C_s,U^C_s)ds|\leq \int_t^T |f(s,Y^C_s,0)|ds+K\int_t^T |U_s|ds\\
	\leq &\int_t^T K(1+|Y_s^C|^\alpha) ds+K\int_t^T |U_s|ds\\
	\leq &\int_t^T K_\alpha(1+|Y_s^C|) ds+K\int_t^T |U_s|ds,
\end{align*}
where $K_\alpha$ is a constant depending on $K,\alpha$. By the Gronwall inequality and Equation \eqref{Y1},  there exists a constant $M>0$, such that for any $t\in[0,T]$
\begin{equation}\label{U1}
	 |U_t^C|\leq M(1+C_T).
\end{equation}
Then, we obtain the uniform boundedness of $Y^{C^n}_t$ and $U^{C^n}_t$ ($t\in[0,T]$, $n\in\mathbb{N}$) by \eqref{Y1} and \eqref{U1}, respectively. By the dominated convergence theorem, we have $U^{C^n}_t\rightarrow U^C_t$ as $n$ goes to infinity, $t\in[0,T]$.

(2) For any $C^i\in \mathcal{A}(w)$, $i=1,2$, let $Y^i$, $U^i$ be the level of satisfaction and utility associated with $C^i$, $i=1,2$, respectively. Then, for any $\lambda \in(0,1)$, set $C^\lambda:=\lambda C^1+(1-\lambda)C^2\in \mathcal{A}(w)$. By Proposition 5 in \cite{DE},  we have for any $t\in[0,T]$, $U^{C^\lambda}_t\geq \lambda U^1_t+(1-\lambda)U^2_t$. Suppose that there exists an interval $(t_0,t_1)$, such that $C^1_s\neq C^2_s$, $s\in(t_0,t_1)$. We then show that $U^{C^\lambda}_0>\lambda U^1_0+(1-\lambda)U^2_0$. For simplicity, we define $\delta_t=U^{C^\lambda}_t-( \lambda U^1_t+(1-\lambda)U^2_t)$. Then, we have
\begin{align*}
\delta_t=&\int_t^T[ f(s,Y^{C^\lambda}_s,U^{C^\lambda}_s)-\lambda f(s,Y_s^1,U_s^1)-(1-\lambda)f(s,Y_s^2,U_s^2)]ds\\
=&\int_t^T [ f(s,Y^{C^\lambda}_s,U^{C^\lambda}_s)-f(s,\lambda Y^1_s+(1-\lambda)Y^2_s, \lambda U^1_s+(1-\lambda)U^2_s)+g_s]ds,
\end{align*}
where $g_s=f(s,\lambda Y^1_s+(1-\lambda)Y^2_s, \lambda U^1_s+(1-\lambda)U^2_s)-\lambda f(s,Y_s^1,U_s^1)-(1-\lambda)f(s,Y_s^2,U_s^2)$. Due to strict concavity, we obtain that $g_s>0$ on some interval $(t'_0,t'_1)$. Since $Y^C$ is linear in $C$, it is easy to check that
\begin{displaymath}
f(s,Y^{C^\lambda}_s,U^{C^\lambda}_s)-f(s,\lambda Y^1_s+(1-\lambda)Y^2_s, \lambda U^1_s+(1-\lambda)U^2_s)\geq -K|\delta_s|.
\end{displaymath}
Applying Lemma \ref{simple}, we then  obtain the strict concavity property.
\end{proof}

Now, we are in a position to state the main result in this section.
		
\begin{theorem}
	Under Assumption \ref{a1}, the utility maximization problem has a solution $C^*$. The solution is unique if $C\rightarrow  Y^C$ is injective.
\end{theorem}

\begin{proof}
	Existence follows from Proposition \ref{p1} and the fact that $\mathcal{A}(w)$ is compact with respect to the weak topology (see Proposition 3.2 in \cite{BR00}). Uniqueness is  due to the strict concavity of the utility functional. In fact, suppose that $C^i$, $i=1,2$ are two different optimal consumption plans. Let $Y^i$, $U^i$ be the level of satisfaction and utility associated with $C^i$, $i=1,2$, respectively. Then, there exists an open interval on which $Y^1\neq Y^2$. It is clear that for any $\lambda \in(0,1)$, $C^\lambda:=\lambda C^1+(1-\lambda)C^2\in \mathcal{A}(w)$ and
	\begin{align*}
	U^{C^\lambda}_0>\lambda U^1_0+(1-\lambda)U^2_0=\sup_{C\in\mathcal{A}(w)}U_0^C,
	\end{align*}
	which leads to a contradiction.
\end{proof}

\begin{remark}
	Suppose that there exist two continuous, positive functions $\theta^i$ defined on $[0,T]$, $i=1,2$, such that $\theta_{t,s}=\theta^1_t\theta^2_s$. Then, the mapping $C\rightarrow  Y^C$ is injective.  A frequently used example   is given by $\theta_{t,s}=\beta e^{-\beta(t-s)}$, $\beta>0$.
\end{remark}

\section{First order conditions for optimality}

In this section, we provide the analogue of the Kuhn-Tucker theorem for the utility maximization problem \eqref{1}. In order to formulate our result, we set
\begin{displaymath}
	\nabla V(C)(t)=\int_t^T \exp\left(\int_0^s \partial_u f(r,Y_r^C,U_r^C)dr\right)\partial_y f(s,Y_s^C,U_s^C)\theta_{s,t}ds.
\end{displaymath}
 In fact, $\nabla V(C)$ can be interpreted as the marginal utility, which is derived from an additional infinitesimal consumption at time $t$, otherwise following the consumption $C$. From the mathematical point of view, $\nabla V(C)$ is the  utility gradient at $C$.

 \begin{remark}
 Recall the optimal consumption problem studied in \cite{SS99}. The Riesz representation of the utility gradient at time $t$ is of the following form
 \begin{displaymath}
	m_t(c)= \exp\left(\int_0^t \partial_u f(r,c_s,V_s(c))dr\right)\partial_y f(s,c_t,V_t(c)),
\end{displaymath}
where $c=\{c_t\}_{t\in[0,T]}$ is the consumption rate and $V_\cdot(c)$ is the recursive utility associated with $c$. Compared with the above result, the gradient at time $t$ in our case depends on the whole path of the consumption, which leads to the main difficulty in the Hindy-Huang-Kreps framework.
 \end{remark}

\begin{theorem}\label{T2}
	Under Assumption \ref{a1}, a consumption plan $C^*\in\mathcal{X}$ solves the utility maximization problem \eqref{1} if and only if the following conditions hold true for some finite Lagrange multiplier $M>0$:
	\begin{description}
		\item[(i)] $\int_0^T \psi_t dC^*_t=w$;
		\item[(ii)] $\nabla V(C^*)(t)\leq M\psi_t$ for any $t\in[0,T]$;
		\item[(iii)] $\int_0^T (\nabla V(C^*)(t)-M\psi_t)dC^*_t=0$.
	\end{description}
\end{theorem}

\begin{proof}
	We first prove the sufficiency. Let $C^*\in\mathcal{X}$ be the consumption plan satisfying conditions (i)-(iii) and $C\in\mathcal{X}$ be another budget feasible consumption plan. For simplicity, set $Y^*=Y^{C^*}$, $Y=Y^C$, $U^*=U^{C^*}$ and $U=U^C$. By the concavity of the {intertemporal aggregator} function $f$ {and the flow property of recursive utility},  for any $0\leq t\leq r\leq T$,
	\begin{align*}
		(U^*_t-U_t)=& \int_t^r (f(s,Y_s^*,U_s^*)-f(s,Y_s,U_s))ds+(U^*_r-U_r)\\
		\geq &\int_t^r \partial_y f(s,Y_s^*,U_s^*)(Y_s^*-Y_s)ds+\int_t^r \partial_u f(s,Y_s^*,U_s^*)(U_s^*-U_s)ds+(U^*_r-U_r)\\
		=&\int_t^r \partial_y f(s,Y_s^*,U_s^*)\int_0^s \theta_{s,t}(dC_t^*-dC_t)ds+\int_t^r \partial_u f(s,Y_s^*,U_s^*)(U_s^*-U_s)ds+(U^*_r-U_r).
	\end{align*}
	By Lemma \ref{gronwall}, we obtain that
	\begin{displaymath}
		U_0^*-U_0\geq \int_0^T \exp\left(\int_0^s \partial_uf(r,Y^*_r,U^*_r)dr\right)\partial_yf(s,Y_s^*,U_s^*)\int_0^s \theta_{s,t}(dC_t^*-dC_t)ds.
	\end{displaymath}
	We divide the last expectation into two terms:
	\begin{displaymath}
		I^*=\int_0^T \exp\left(\int_0^s \partial_uf(r,Y^*_r,U^*_r)dr\right)\partial_yf(s,Y_s^*,U_s^*)\int_0^s \theta_{s,t}dC_t^*ds
	\end{displaymath}
	and
		\begin{displaymath}
	I=\int_0^T \exp\left(\int_0^s \partial_uf(r,Y^*_r,U^*_r)dr\right)\partial_yf(s,Y_s^*,U_s^*)\int_0^s \theta_{s,t}dC_tds.
	\end{displaymath}
	By simple calculation, we derive that
	\begin{align*}
		I^*&=\int_0^T \int_t^T \exp\left(\int_0^s \partial_uf(r,Y^*_r,U^*_r)dr\right)\partial_yf(s,Y_s^*,U_s^*)\theta_{s,t}ds dC_t^*\\
		&=\int_0^T \nabla V(C^*)(t)dC_t^*=M \int_0^T \psi_tdC_t^*=M w,
	\end{align*}
	where we use the Fubini theorem in the first equality. A similar analysis yields that
	\begin{displaymath}
		I=\int_0^T \nabla V(C^*)(t)dC_t\leq M \int_0^T \psi_tdC_t\leq M w.
	\end{displaymath}
	Combining the above results implies that $U^*_0-U_0\geq 0$. Hence, $C^*$ is optimal. Necessity follows from Lemma \ref{l4} below and Lemma 4.4 in \cite{BR00}.
\end{proof}

The following lemma indicates that if $C^*$ is optimal for the original problem, it also solves a suitable linear utility maximization problem, whose solution is characterized by Lemma 4.4 in \cite{BR00}.
\begin{lemma}\label{l4}
	Let $C^*$ be optimal for the original problem \eqref{1} and let $\phi^*=\nabla V(C^*)$.  Then, $C^*$ is also optimal for the following linear problem
	\begin{equation}\label{6}
		\sup_{C\in\mathcal{A}(w)} \int_0^T \phi^*_t dC_t.
	\end{equation}
	Furthermore, the value of this problem is finite.
\end{lemma}

\begin{proof}
	For any $C\in\mathcal{A}(w)$ and $\varepsilon\in[0,1]$, let $C^\varepsilon=\varepsilon C+(1-\varepsilon)C^*$. For simplicity, set $Y^*=Y^{C^*}$, $Y=Y^C$, $Y^\varepsilon=Y^{C^\varepsilon}$, $U^*=U^{C^*}$, $U=U^C$ and $U^\varepsilon=U^{C^\varepsilon}$. By the concavity of the felicity function, we have, for any $0\leq t<r\leq T$,
	\begin{align*}
	\frac{1}{\varepsilon}(U^\varepsilon_t-U_t^*)
	=& \frac{1}{\varepsilon}\int_t^r (f(s,Y_s^\varepsilon,U_s^\varepsilon)-f(s,Y_s^*,U_s^*))ds+\frac{1}{\varepsilon}(U^\varepsilon_r-U_r^*)\\
	\geq &\frac{1}{\varepsilon}\int_t^r \partial_y f(s,Y_s^\varepsilon,U_s^\varepsilon)(Y_s^\varepsilon-Y_s^*)ds+\int_t^r \partial_u f(s,Y_s^\varepsilon,U_s^\varepsilon)(U_s^\varepsilon-U_s^*)ds+\frac{1}{\varepsilon}(U^\varepsilon_r-U_r^*)\\
	=&\int_t^r \partial_y f(s,Y_s^\varepsilon,U_s^\varepsilon)\int_0^s \theta_{s,t}(dC_t-dC_t^*)ds\\
	 &+\int_t^r \partial_u f(s,Y_s^\varepsilon,U_s^\varepsilon)\frac{1}{\varepsilon}(U_s^\varepsilon-U_s^*)ds+\frac{1}{\varepsilon}(U^\varepsilon_r-U_r^*),
	\end{align*}
	where the last equality follows from the fact that
	\begin{displaymath}
		Y_s^\varepsilon-Y_s^*=\varepsilon (Y_s-Y_s^*)=\varepsilon \int_0^s \theta_{s,t}(dC_t-dC_t^*).
	\end{displaymath}
	By Lemma \ref{gronwall} and noting that $C^*$ is optimal for the utility maximization problem \eqref{1}, we obtain that
	\begin{displaymath}
		0\geq \frac{1}{\varepsilon}(U^\varepsilon_0-U^*_0)\geq \int_0^T \exp\left(\int_0^s \partial_uf(r,Y_r^\varepsilon,U_r^\varepsilon)dr\right)\partial_yf(s,Y_s^\varepsilon,U_s^\varepsilon)\int_0^s \theta_{s,t}(dC_t-dC^*_t)ds.
	\end{displaymath}
	Set $\Phi^\varepsilon(t)=\int_t^T \exp(\int_0^s \partial_uf(r,Y_r^\varepsilon,U_r^\varepsilon)dr)\partial_yf(s,Y_s^\varepsilon,U_s^\varepsilon)\theta_{s,t}ds$, $t\in[0,T]$. By the Fubini theorem, the above equation yields that, for any $\varepsilon\in(0,1]$,
	\begin{equation}\label{varepsilon}
		\int_0^T \Phi^\varepsilon(t)dC_t^*\geq \int_0^T \Phi^\varepsilon(t)dC_t.
	\end{equation}
	By the continuity of the utility function and the level of satisfaction w.r.t consumption, it follows that
	\begin{displaymath}
	   \Phi^*(t):=\Phi^0(t)=\lim_{\varepsilon\rightarrow 0}\Phi^\varepsilon(t), \ t\in[0,T].
	\end{displaymath}
	Letting $\varepsilon$ go to $0$ in Eq. \eqref{varepsilon}, by the dominated convergence theorem, we obtain that
    $$\int_0^T \Phi^*(t)dC_t^*\geq \int_0^T \Phi^*(t)dC_t,$$
    which is the desired result.
\end{proof}

By Theorem \ref{T2}, it is easy to obtain the following result.
   \begin{corollary}\label{cor4.5}
   Every solution $C^*\in\mathcal{A}(w)$ of the utility maximization problem \eqref{1} satisfies
   \begin{itemize}
   \item[(i)] $\int_0^T \psi_t dC^*_t=w$;
   \item[(ii)] $\supp dC^*\subset \arg\max \frac{\nabla V(C^*)}{\psi}.$
   \end{itemize}
   \end{corollary}

\section{Construction of optimal consumption plans}

In the previous section, we have shown that the optimal consumption plan for problem \eqref{1} exists.  In this section, we are going to provide a more explicit method to construct the optimal consumption plan with the help of the first order conditions obtained in Theorem \ref{T2} and Corollary \ref{cor4.5}. To this end, we make the following assumptions for the price $\psi$, the functions $\eta,\theta$ and the felicity function $f$.

\begin{assumption}\label{a5.1}
(i) The price function $\psi$ is given by
 \begin{displaymath}
 \psi_t=e^{-r t}
 \end{displaymath}
for some constant interest rate $r\geq 0$. The level of satisfaction $Y^C$ is given by
\begin{displaymath}
Y^C_t=y e^{-\beta t}+\beta\int_0^t e^{-\beta(t-s)}dC_s
\end{displaymath}
for some $y,\beta>0$.

(ii) The felicity function $f$ and its first and second partial derivatives are continuous on $[0,T]\times(0,\infty)\times \mathbb{R}$. Furthermore, for any $(t,u)\in[0,T]\times\mathbb{R}$, we have
\begin{displaymath}
\lim_{y\downarrow 0}\partial_y f(t,y,u)=+\infty, \ \lim_{y\uparrow +\infty}\partial_y f(t,y, u)=0.
\end{displaymath}
Finally, for any $(t,y,u)\in[0,T]\times (0,\infty)\times\mathbb{R}$, we have $\mathfrak{L}f(t,y,u)>0$, where
\begin{align*}
\mathfrak{L}f(t,y,u)=&r \partial_y f(t,y,u)+\partial_u f(t,y,u)\partial_y f(t,y,u)+\partial_{ty} f(t,y,u)\\
&-\beta y \partial_y^2 f(t,y,u)-f(t,y,u)\partial_{uy}f(t,y,u).
\end{align*}
\end{assumption}

\begin{remark}\label{remarkf}
(i) 
Consider the time-additive case, where $f(t,y,u)=g(y)-\delta u$. The corresponding instantaneous utility function is given by $h(t,y)=e^{-\delta t} g(y)$. It is easy to check that $\mathfrak{L}f(t,y,u)>0$ is equivalent to $\mathcal{L}h(t,y)>0$, where $\mathcal{L}$ is the differential operator defined by $\mathcal{L}:=r\partial_y-\beta y \partial_y^2+\partial_{ty}$ in \cite{BR00}. That is to say, Assumption \ref{a5.1} (ii) coincides with Assumption 5.1 (ii) in \cite{BR00} in the time-additive case.

(ii) Consider the Epstein-Zin felicity function $f(t,y,u)=\frac{\delta}{1-\frac{1}{\alpha}}y^{1-\frac{1}{\alpha}} [(1-\rho) u]^{1-\frac{1}{\psi}}-\delta \psi u$ with $\delta\geq 0$, $0<\rho\neq 1$ and $0<\alpha\neq 1$, where
 $$\psi= \frac{1-\rho}{1-\frac{1}{\alpha}}.$$
  Here, $\rho$ represents the agent's relative risk aversion, $\alpha$ is his elasticity of intertemporal substitution and $\delta$ is his rate of time preference. Assumption \ref{a5.1} (ii) is equivalent to the requirement $r+\frac{1}{\alpha}\beta-\delta>0$.
\end{remark}

\subsection{Constructive method for optimal consumption plan}

In this subsection, we try to use the first order conditions in Theorem \ref{T2} to give an explicit description of the optimal consumption plan. The first problem is that one does not know enough about the support of optimal consumption plan. We will show that, in fact, Assumption \ref{a5.1} ensures that the support of optimal consumption plan is a closed interval (see Lemma \ref{lem5.6}). Therefore, for a fixed  Lagrange multiplier $M>0$, we guess that the level of satisfaction associated with optimal consumption plan $C^M$ is of the following form:

\begin{equation}\label{resulty}
{Y}^M_t=\begin{cases}
y e^{-\beta t}, &t\in[0,t_0(M));\\
y e^{-\beta t}+\int_{t_0(M)}^{t} \beta e^{-\beta(t-s)}dC^M_s, &t\in[t_0(M),t_1(M)];\\
(y e^{-\beta t_1(M)}+\int_{t_0(M)}^{t_1(M)} \beta e^{-\beta(t_1(M)-s)}dC^M_s) e^{-\beta(t-t_1(M))}, &t\in(t_1(M),T],
\end{cases}
\end{equation}
where the agent refrains from consumption outside the interval  $(t_0(M)$, $t_1(M))$. We set $I_t^M=y e^{-\beta t}+\int_{t_0(M)}^{t} \beta e^{-\beta(t-s)}dC^M_s$. Let $U^M$ be the utility corresponding to $Y^M$, i.e.,
\begin{equation}\label{UM}
U^M_t=\int_t^T f(s,Y^M_s,U^M_s)ds.
\end{equation}
 We define
\begin{equation}\label{PhiM}
\Phi^M_t=\beta e^{(r+\beta)t}\int_t^T \exp\left(\int_0^s \partial_u f(r,Y^M_r,U^M_r)dr\right) \partial_y f(s,Y^M_s,U^M_s)e^{-\beta s}ds.
\end{equation}
Based on the first order conditions obtained in Theorem \ref{T2}, $t_0(M),t_1(M), C^M$ should satisfy the following condition:
\begin{equation}\label{verification}
\begin{cases}
\Phi^M_t< M, \ t\in[0,t_0(M))\cup(t_1(M),T];\\
\Phi^M_t= M, \ t\in[t_0(M),t_1(M)].
\end{cases}
\end{equation}

\begin{theorem}\label{construction}
Under Assumption \ref{a5.1}, suppose that for each fixed $M>0$, there exists a unique triple $(t_0(M),t_1(M),C^M)$ satisfying \eqref{verification}. Then, we have
\begin{equation}\label{CM}
d C^M_t=-\frac{\mathfrak{L}f(t,Y^M_t,U^M_t)}{\beta \partial_y^2 f(t,Y^M_t,U^M_t)}I_{(t_0(M),t_1(M))}(t) dt.
\end{equation}
 The solution of the optimal consumption problem \eqref{1} is $C^{M^*}$ with $M^*$ the solution to the equation $\int_0^T e^{-rt}dC^{M^*}_t=w$.
\end{theorem}

\begin{proof}
Since $(t_0(M),t_1(M),C^M)$ is a solution to \eqref{verification}, then for any $t\in(t_0(M),t_1(M))$, we have
\begin{displaymath}
\beta e^{(r+\beta)t}\int_t^T \exp\left(\int_0^s \partial_u f(r,Y^M_r,U^M_r)dr\right) \partial_y f(s,Y^M_s,U^M_s)e^{-\beta s}ds=M.
\end{displaymath}
Taking derivative w.r.t. $t$ yields that
\begin{displaymath}\label{firstorder}
(r+\beta)M-\beta e^{rt}\exp\left(\int_0^t \partial_u f(r,Y^M_r,U^M_r)dr\right) \partial_y f(t,Y^M_t,U^M_t)=0.
\end{displaymath}
Taking derivative w.r.t $t$ again, we have
\begin{displaymath}\begin{split}
0=&r(r+\beta)M+(r+\beta)M\partial_u f(t,Y^M_t,U^M_t)+(r+\beta)M\frac{\partial_{ty} f(t,Y^M_t,U^M_t)}{\partial_y f(t,Y^M_t,U^M_t)}\\
&+(r+\beta)M\frac{\partial_y^2 f(t,Y^M_t,U^M_t)}{\partial_y f(t,Y^M_t,U^M_t)} \frac{dY^M_t}{dt}-(r+\beta)M\frac{ f(t,Y^M_t,U^M_t)\partial_{uy} f(t,Y^M_t,U^M_t)}{\partial_y f(t,Y^M_t,U^M_t)}.
\end{split}\end{displaymath}
Simple calculation implies that
\begin{align}\label{YM}
dY_t^M=-\frac{\mathfrak{L}(t,Y^M_t,U^M_t)+\beta Y^M_t\partial_y^2 f(t,Y^M_t,U^M_t)} {\partial_y^2 f(t,Y^M_t,U^M_t)}dt. 
\end{align}
By the dynamic of $Y^{C}$ in Assumption \ref{a5.1}, we know that
\begin{displaymath}
dC^M_t=\frac{1}{\beta}d Y^M_t+Y^M_t dt.
\end{displaymath}
Plugging $dY^M_t$ obtained by \eqref{YM} into the above equation yields the desired result.
\end{proof}

\begin{remark}
Suppose that for each $M>0$, there exists a unique triple $(t_0(M),t_1(M),C^M)$ satisfying \eqref{verification}. Then, the consumption $C^M$ occurs in rates.  Besides, for the time-additive case with $f(t,y,u)=g(y)-\delta u$, Equation \eqref{CM} degenerates to Equation (11) in \cite{BR00}.
\end{remark}

Finding $t_0(M),t_1(M)$ and $C^M$  requires solving a fully coupled forward-backward system \eqref{resulty}-\eqref{verification}. We solve this problem for some typical cases explicitly in the next section.

\subsection{Explicit solutions for the Epstein-Zin felicity function}

In this subsection, we provide the explicit solution to the optimal consumption problem \eqref{1} when the felicity function is of the Epstein-Zin type as in Remark \ref{remarkf}. More precisely, $f$ takes the following form
 \begin{equation}\label{epsteinzin}
 f(t,y,u)=\frac{\delta}{1-\frac{1}{\alpha}}y^{1-\frac{1}{\alpha}} [(1-\rho) u]^{1-\frac{1}{\psi}}-\delta \psi u,
 \end{equation}
  where
 $$\psi= \frac{1-\rho}{1-\frac{1}{\alpha}}$$
and $\delta, \rho,\alpha$ are positive constants  satisfying $0<\rho\neq1$, $\rho<\frac{1}{\alpha}\neq1$ and $r+\frac{1}{\alpha}\beta-\delta>0$.

 \begin{remark}\label{assumptionf}
 Assumption \ref{a5.1} (ii) is satisfied when we have $0<\rho\neq1$, $\rho<\frac{1}{\alpha}\neq1$ and $r+\frac{1}{\alpha}\beta-\delta>0$. Besides, if $\rho\alpha=1$,  our utility function $U$ coincides with the power felicity function (12) in \cite{BR00},  with $\alpha$ replaced by $1-1/\alpha$.
 \end{remark}

Recalling the previous subsection, we need to solve the following equation for $t\in[t_0(M),t_1(M)]$,
\begin{displaymath}
\int_t^T \exp\left(\int_0^s \partial_u f(r,Y^M_r,U^M_r)dr\right) \partial_y f(s,Y^M_s,U^M_s) \beta e^{\beta(t-s)} ds=Me^{-rt},
\end{displaymath}
where  $t_0(M)$ and  $t_1(M)$ are the times when consumption starts resp. stops, respectively. 
Taking derivatives with respect to $t$, we obtain that
\begin{equation}\label{firstorder}
(r+\beta)M-\beta e^{rt}\exp\left(\int_0^t \partial_u f(r,Y^M_r,U^M_r)dr\right) \partial_y f(t,Y^M_t,U^M_t)=0.
\end{equation}
Simple calculations imply
\begin{align*}
&\partial_y f(t,Y^M_t,U^M_t)=\delta e^{\delta(\psi-1)t}(Y^M_t)^{-\frac{1}{\alpha}}\left(\int_t^T \delta e^{-\delta s}(Y_s^M)^{1-\frac{1}{\alpha}} dr\right)^{\psi-1},\\
&\partial_u f(t,Y^M_t,U^M_t)=(\psi-1)\frac{\delta e^{-\delta t}(Y_s^M)^{1-\frac{1}{\alpha}} }{\int_t^T \delta e^{-\delta s} (Y_s^M)^{1-\frac{1}{\alpha}}  ds}-\delta \psi.
\end{align*}
Hence,
\begin{align*}
\int_0^t \partial_u f(r,Y^M_r,U^M_r)dr=(\psi-1)\left(\ln\left(\int_0^T \delta e^{-\delta s} (Y^M_s)^{1-\frac{1}{\alpha}} ds\right)-\ln\left(\int_t^T \delta e^{-\delta s}(Y_s^M)^{1-\frac{1}{\alpha}}  ds\right)\right)-\delta \psi t.
\end{align*}
Plugging the above results into Eq. \eqref{firstorder}, we obtain that
\begin{displaymath}
(r+\beta)M-\beta e^{(r-\delta)t}\delta\left(\int_0^T\delta e^{-\delta s}(Y_s^M)^{1-\frac{1}{\alpha}}  ds\right)^{\psi-1}(Y_t^M)^{-\frac{1}{\alpha}}=0.
\end{displaymath}
Set $A_M=\delta(\int_0^T\delta e^{-\delta s}(Y_s^M)^{1-\frac{1}{\alpha}} ds)^{\psi-1}$. Then, we derive that for any $t\in[t_0(M),t_1(M)]$
\begin{displaymath}
Y_t^M=A_M^{\alpha} \left(\frac{M(r+\beta)}{\beta}\right)^{-\alpha} e^{-\alpha(\delta- r)t}.
\end{displaymath}
Finally, we have
\begin{equation}\label{result y}
{Y}^M_t=\begin{cases}
y e^{-\beta t}, &t\in[0,t_0(M));\\
A_M^{\alpha} \left(\frac{M(r+\beta)}{\beta}\right)^{-\alpha} e^{-\alpha(\delta- r)t}, &t\in[t_0(M),t_1(M)];\\
A_M^{\alpha} \left(\frac{M(r+\beta)}{\beta}\right)^{-\alpha} e^{-\alpha(\delta- r)t_1(M)}e^{-\beta(t-t_1(M))}, &t\in(t_1(M),T].
\end{cases}
\end{equation}

The remaining problem is to calculate $t_0(M),t_1(M)$ and $A_M$ (or $Y^M$, equivalently) explicitly. Actually, $Y^M,t_0(M),t_1(M)$ should satisfy \eqref{verification} with
$$A_M=\delta\left(\int_0^T\delta e^{-\delta s}(Y^M_s)^{1-\frac{1}{\alpha}} ds\right)^{\psi-1}.$$
Finally, we need to find the appropriate $M$ such that $\int_0^T e^{-rt}dC_t^M=\int_0^T\frac{1}{\beta} e^{-rt}(dY^M_t+\beta Y^M_t dt)=w$, where $C^M$ is the consumption plan resulting the level of satisfaction $Y^M$.

\begin{theorem}\label{final}
Suppose that $\bar{\tau}>0$, where
\begin{equation}\label{bartau}
\bar{\tau}=\begin{cases}
T-\frac{1}{\delta+\beta(1-\frac{1}{\alpha})}\ln\frac{r+\beta}{r+\frac{\beta}{\alpha}-\delta}, & \delta+\beta(1-\frac{1}{\alpha})\neq 0,\\
T-\frac{1}{r+\beta}, &\textrm{otherwise}.
\end{cases}
\end{equation}
Let
\begin{displaymath}
k^*=\begin{cases}
\frac{r+\frac{\beta}{\alpha}-\delta}{\beta(\delta-(1-\frac{1}{\alpha})r)}(1-e^{-[\alpha\delta+(1-\alpha)r]\bar{\tau}}), & \delta\neq (1-\frac{1}{\alpha})r,\\
(1-\frac{\alpha(\delta-r)}{\beta})\bar{\tau}, & \textrm{otherwise}.
\end{cases}
\end{displaymath}

If $w\geq k^* y$, it is optimal to have an initial consumption gulp of size $\Delta C_0=(w-k^* y)/(1+\beta k^*)$ and then for $t\in(0,\bar{\tau}]$, to consume at rates
\begin{equation}\label{optimalc1}
dC_t=\left(1-\frac{\alpha(\delta-r)}{\beta}\right)\frac{y+\beta w}{1+\beta k^*}e^{-\alpha(\delta-r)t}dt.
\end{equation}

If $w<k^* y$, the investor optimally waits until time
\begin{displaymath}
\underline{\tau}=\frac{1}{r+\frac{\beta}{\alpha}-\delta}\ln \left (M^* \frac{r+\beta}{\beta} y^{\frac{1}{\alpha}}\right)
\end{displaymath}
then he starts to consume at rates
\begin{equation}\label{optimalc2}
dC_t=\frac{\beta+\alpha(r-\delta)}{\beta}\left(M^* \frac{r+\beta}{\beta}\right)^{-\alpha}e^{-\alpha(\delta-r)t}dt
\end{equation}
until time $\bar{\tau}$.
Here, $M^*=(\beta K^*)/(r+\beta)$, where $K^*>0$ is the unique solution to
\begin{displaymath}
K^{-\frac{r+\beta}{r+\frac{\beta}{\alpha}-\delta}} y^{-\frac{\alpha \delta-(\alpha-1)r}{\beta+\alpha(r-\delta)}}-K^{-\alpha} e^{-[\alpha\delta-(\alpha-1)r]\bar{\tau}}=\frac{\beta[\alpha \delta-(\alpha-1)r]}{\beta+\alpha(r-\delta)}w,
\end{displaymath}
if $\delta\neq (1-\frac{1}{\alpha})r$, or to
\begin{displaymath}
K^-\alpha\left( \frac{1}{\beta}\ln K+\frac{1}{\alpha\beta}\ln y-(r+\frac{\beta}{\alpha}-\delta)\bar{\tau} \right)=-\alpha w,
\end{displaymath}
if $\delta= (1-\frac{1}{\alpha})r$.
\end{theorem}

\begin{proof}
Recalling \eqref{PhiM} and noting that
\begin{displaymath}
U_t=\frac{1}{1-\rho}\left(\int_t^T  \delta e^{\delta (t-s)}Y_s^{1-\frac{1}{\alpha}} ds\right)^{\psi},
\end{displaymath}
we may obtain
\begin{displaymath}
\Phi_t=\left(\int_0^T\delta e^{-\delta s}Y^{1-\frac{1}{\alpha}} _s ds\right)^{\psi-1}e^{(r+\beta)t}\int_t^T \delta \beta e^{-(\beta+\delta)s} Y_s^{-\frac{1}{\alpha}} ds.
\end{displaymath}
By Theorem \ref{construction}, it is sufficient to prove $\int_0^T e^{-rt}dC_t=w$ and
\begin{equation}\label{verification1}
\begin{cases}
\Phi_t< M_1, \ t\in(\bar{\tau},T];\\
\Phi_t= M_1, \ t\in[0,\bar{\tau}],
\end{cases}
\end{equation}
for the case that $w\geq k^* y$, or
\begin{equation}\label{verification2}
\begin{cases}
\Phi_t< M_2, \ t\in[0,\underline{\tau})\cup(\bar{\tau},T];\\
\Phi_t= M_2, \ t\in[\underline{\tau},\bar{\tau}],
\end{cases}
\end{equation}
for the case that $w<k^* y$, where $M_1$, $M_2$ are positive constants.

 \textbf{Case 1:} $w\geq k^* y$. We may get the associated level of satisfaction as
\begin{equation}\label{optimaly1}
Y_t=\begin{cases}
\frac{y+\beta w}{1+\beta k^*}e^{-\alpha(\delta-r)t},  & t\in[0,\bar{\tau}],\\
\frac{y+\beta w}{1+\beta k^*}e^{-\alpha(\delta-r)\bar{\tau}}e^{-\beta(t-\bar{\tau})}, & t\in(\bar{\tau},T].
\end{cases}
\end{equation}
By simple calculations, we could get that for any $t\in[0,\bar{\tau}]$,
\begin{displaymath}
\Phi_{t}=M
\end{displaymath}
and for any $t\in(\bar{\tau},T]$,
\begin{displaymath}
\Phi_t=M(r+\beta)e^{(r+\beta)(t-\bar{\tau})}e^{[\delta+\beta(1-\frac{1}{\alpha})]\bar{\tau}}\frac{1}{\delta+\beta(1-\frac{1}{\alpha})}\left(e^{-[\delta+\beta(1-\frac{1}{\alpha})]t}-e^{-[\delta+\beta(1-\frac{1}{\alpha})T]}\right),
\end{displaymath}
where
\begin{displaymath}
M=\left(\int_0^T\delta e^{-\delta s}Y^{1-\frac{1}{\alpha}} _s ds\right)^{\psi-1}\frac{\beta\delta}{r+\beta}\left(\frac{y+\beta w}{1+\beta k^*}\right)^{-1/\alpha}.
\end{displaymath}
It is easy to check that $\Phi'_t<0$ if $t>\bar{\tau}$. Hence, for $t\in(\bar{\tau},T]$, we have
$$\Phi_t<\Phi_{\bar{\tau}}=M.$$
The above analysis indicates that conditions \eqref{verification1} are satisfied. Besides, we may check that $C$ given by \eqref{optimalc1} satisfies $\int_0^T e^{-rs}dC_s=w$. Hence, $C$ is the optimal consumption plan.

\textbf{Case 2:} $w<k^* y$. We may get the associated level of satisfaction as
\begin{equation}\label{optimaly2}
Y_t=\begin{cases}
ye^{-\beta t}, &t\in[0,\underline{\tau}),\\
(M^*\frac{r+\beta}{\beta})^{-\alpha} e^{-\alpha(\delta-r)t},  & t\in[\underline{\tau},\bar{\tau}],\\
(M^*\frac{r+\beta}{\beta})^{-\alpha} e^{-\alpha(\delta-r)\bar{\tau}}e^{-\beta(t-\bar{\tau})}, & t\in(\bar{\tau},T].
\end{cases}
\end{equation}
By simple calculations, we could get that for any $t\in[\underline{\tau},\bar{\tau}]$,
\begin{displaymath}
\Phi_{t}=M,
\end{displaymath}
where $M=\delta M^*(\int_0^T\delta e^{-\delta s}Y^{1-\frac{1}{\alpha}} _s ds)^{\psi-1}$.
When $t\in(\bar{\tau},T]$, we obtain that
\begin{displaymath}
\Phi_t= M(r+\beta) e^{[\delta-r-\frac{\beta}{\alpha}]\bar{\tau}}
\frac{1}{\delta+\beta(1-\frac{1}{\alpha})}\left[ e^{(r-\delta+\frac{\beta}{\alpha})t}-e^{-[\delta+\beta(1-\frac{1}{\alpha})]T+(r+\beta)t}\right].
\end{displaymath}
We may check that $\Phi'_t<0$ if $t>\bar{\tau}$. Hence, for $t\in(\bar{\tau},T]$, we have
$$\Phi_t<\Phi_{\bar{\tau}}=M.$$
For the case that $t\in[0,\underline{\tau})$, simple calculation yields that
\begin{displaymath}
\Phi_t=\frac{M}{M^*}\left[M^* e^{-(r+\beta)(t-\underline{\tau})}+\frac{\beta y^{-\frac{1}{\alpha}}}{\delta+\beta(1-\frac{1}{\alpha})}   \left(e^{-(r+\frac{\beta}{\alpha}-\delta)t}-e^{-[\delta+\beta(1-\frac{1}{\alpha})]\underline{\tau}+(r+\beta)t}\right)  \right].
\end{displaymath}
We may check that $\Phi'_t>0$ if $t<\underline{\tau}$. Hence, for $t\in[0,\underline{\tau})$, we have
$$\Phi_t<\Phi_{\underline{\tau}}=M.$$
All above analysis indicates that conditions \eqref{verification2} are satisfied. Besides, we may check that $C$ given by \eqref{optimalc2} satisfies $\int_0^T e^{-rs}dC_s=w$. Hence, $C$ given by \eqref{optimalc2} is the optimal consumption plan.
\end{proof}

\begin{remark}
Clearly, in the deterministic setting, the optimal consumption plan does not depend on the relative risk aversion coefficient $\rho$. \end{remark}

\begin{remark}
	Given the level of satisfaction $Y$, the recursive utility induced by $Y$ can be written as
	\begin{equation}\label{U}
		U_t=\frac{1}{1-\rho}\left(\int_t^T  \delta e^{\delta (t-s)}Y_s^{1-\frac{1}{\alpha}} ds\right)^{\psi}.
	\end{equation}
It is clear that the optimal consumption plan maximizing $U_0$ over all $C\in\mathcal{A}(w)$ is equivalent to the optimal consumption plan maximizing $\widetilde{U}_0$, where
\begin{displaymath}
	\widetilde{U}_0=\int_0^T  \frac{1}{1-\frac{1}{\alpha}}e^{\delta (t-s)}Y_s^{1-\frac{1}{\alpha}} ds.
\end{displaymath}
The latter problem has been studied  in \cite{BR00} (see Theorem 5.4 with $\alpha$ replacing by $1-1/\alpha$). Therefore, the optimal consumption plan in the recursive setting is the same with the one in \cite{BR00}. That is to say, the solution of utility maximization problem when the felicity function is of Epstein-Zin's type coincides with the one for the time additive utility maximization problem, which is also indicated in \cite{SS99}.
\end{remark}

\begin{remark}
If Assumption \ref{assumptionf} does not hold or $\bar{\tau}$ defined by \eqref{bartau} is nonpositive, then the optimal consumption plan $C^*$ is to consume the whole wealth $w$ at time $0$. In fact, simple calculation implies that $\Phi^*_t:=\psi_t^{-1}\nabla V(C^*)(t)$ is strictly decreasing in $t$. Thus, by the first order condition obtained in Theorem \ref{T2}, $C^*$ is indeed optimal.
\end{remark}

\subsection{Approximation of the utility}

In this section, we consider the approximation of the utility $U$ in (\ref{3}). Let us consider the following dynamic structure with $0\leq s\leq t\leq T$,
\begin{equation}\label{dpp-y}
dY^{C,s,y}_t=\beta(-Y^{C,s,y}_tdt+dC_t)
\end{equation}
with $Y^{C,s,y}_s=y,\beta>0$, and
\begin{displaymath}\label{dpp-w}
 \psi_t=e^{-r (t-s)}.
 \end{displaymath}
Furthermore, we introduce a dynamic constraint structure, where
\begin{equation}
dX^{C,s,w}_t=rX^{C,s,w}_tdt-dC_t,\ X^{C,s,w}_s=w.
\end{equation}
Then, it follows that
$$
X^{C,s,w}_t=e^{r(t-s)}\left(w-\int_s^te^{-r(h-s)}dC_h\right).
$$

Let $\mathcal{K}[s,t]$ denote the set of nonnegative, nondecreasing and right-continuous functions on $[s,t]$,  and thus
$$
\mathcal{A}[s,t]=\{C\in\mathcal{K}[s,t],\ \int_s^te^{-r(h-s)}dC_h\leq w \}
$$
is equal to $\mathcal{B}[s,t:w,y]$, which satisfies
$$
\mathcal{B}[s,t:w,y]=\{C:\ C\in \mathcal{K}[s,t], Y^C_s=y,\ X^{C,s,w}_s=w, X^{C,s,w}_r\geq 0,\ s\leq r\leq t\}.
$$

The agent's recursive utility of the HHK type is of the form
\begin{equation}\label{dpp-0}
U^C_s=\int_s^T f(t,Y_t^{C,w,y},U^C_t)dt,\ C\in \mathcal{A}[s,T],
\end{equation}
and agent wants to maximum the recursive utility at time $0$,
$$
U(0,w,y)=\sup_{C\in \mathcal{A}[0,T]}U^C_0,
$$
which equals to
\begin{equation}\label{dpp-1}
U(0,w,y)=\sup_{C\in \mathcal{B}[0,T:w,y]}\int_0^T f(t,Y_t^{C,w,y},U^C_t)dt.
\end{equation}

Based on the new formula (\ref{dpp-1}), the related classical dynamic programming principle is given as follows. For notations simplicity, we omit the initial states of $X^{C,s,w}$ and $Y^{C,s,y}$ and denote by $X^{C}$ and $Y^{C}$.
\begin{theorem}\label{dpp-thm}
Let Assumption \ref{a1} hold, we have that
\begin{equation}\label{dpp-3}
U(s,x,y)=\sup_{C\in \mathcal{B}[s,t:w,y]}\left[\int_s^t f(r,Y_r^C,U(r,X_r^C,Y_r^C))dr+U(t,X_t^C,Y_t^C)\right].
\end{equation}
\end{theorem}

Based on  dynamic programming principle (\ref{dpp-3}), we consider the following iteration equations,
\begin{equation}\label{dpp-4}
U^{(n)}(s,x,y)=\sup_{C\in \mathcal{B}[s,T:x,y]}\left[\int_s^T f(r,Y_r^C,U^{(n-1)}(r,X_r^C,Y_r^C))dr\right],\ n\geq 1,
\end{equation}
with $U^{(0)}=0$.
\begin{theorem}\label{dpp-lem}
Let Assumption \ref{a1} hold, we have that sequence $\{U^{(n)}(\cdot)\}$ uniformly converges on $[0,T]$.
\end{theorem}
\begin{proof}
Applying the Lipschitz  condition of $f$ on $U$, we have that
\begin{align*}
&\left|U^{(n+1)}(s,x,y)-U^{(n)}(s,x,y)\right|\\
=& \left|\sup_{C\in \mathcal{B}[s,T:x,y]}\left[\int_s^T f(r,Y_r^C,U^{(n)}(r,X_r^C,Y_r^C))dr\right]-\sup_{C\in \mathcal{B}[s,T:x,y]}\left[\int_s^T f(r,Y_r^C,U^{(n-1)}(r,X_r^C,Y_r^C))dr\right]\right|              \\
\leq &\sup_{C\in \mathcal{B}[s,T:x,y]}  \left|\int_s^T \left[f(r,Y_r^C,U^{(n)}(r,X_r^C,Y_r^C))dr-f(r,Y_r^C,U^{(n-1)}(r,X_r^C,Y_r^C))\right]dr\right|               \\
\leq& \sup_{C\in \mathcal{B}[s,T:x,y]}\int_s^TL\left|U^{(n)}(r,X_r^C,Y_r^C)-U^{(n-1)}(r,X_r^C,Y_r^C)\right|dr \\
\leq &\sup_{C\in \mathcal{B}[s,T:x,y]}\sup_{s\leq r\leq T}\left|U^{(n)}(r,X_r^C,Y_r^C)-U^{(n-1)}(r,X_r^C,Y_r^C)\right|L(T-s),
\end{align*}
where $L$ is the Lipschitz constant of $f$ on $U$, which deduces that, for any given $C\in \mathcal{B}[s,T]$ and $r\in [s,T]$,
\begin{align*}
&\left|U^{(n+1)}(r,X_r^C,Y_r^C)-U^{(n)}(r,X_r^C,Y_r^C)\right|\\
\leq &\sup_{C\in \mathcal{B}[s,T:x,y]}\sup_{s\leq r\leq T}\left|U^{(n)}(r,X_r^C,Y_r^C)-U^{(n-1)}(r,X_r^C,Y_r^C)\right|L(T-s).
\end{align*}
and thus
\begin{align*}
&\sup_{C\in \mathcal{B}[s,T:x,y]}\sup_{s\leq r\leq T}\left|U^{(n+1)}(r,X_r^C,Y_r^C)-U^{(n)}(r,X_r^C,Y_r^C)\right|\\
\leq &\sup_{C\in \mathcal{B}[s,T:x,y]}\sup_{s\leq r\leq T}\left|U^{(n)}(r,X_r^C,Y_r^C)-U^{(n-1)}(r,X_r^C,Y_r^C)\right|L(T-s).
\end{align*}
Then, let $s$ satisfy $L(T-s)<1$, we have $U^{(n)}(\cdot)$ converges on $[s,T]$.

Now, we consider the time interval $[s_0,s]$ which satisfies $L(s-s_0)<1$. For initial state $(s_0,x_0,y_0)$,
$$
U^{(n)}(s_0,x_0,y_0)=\sup_{C\in \mathcal{B}[s_0,s:x_0,y_0]}\left[\int_{s_0}^s f(r,Y_r^C,U^{(n-1)}(r,X_r^C,Y_r^C))dr+U^{(n)}(s,X_s^C,Y_s^C)\right],\ n\geq 1.
$$
Using a similar manner on time interval $[s,T]$ case, we can show that,
\begin{align*}
&\sup_{C\in \mathcal{B}[s_0,s:x_0,y_0]}\sup_{s_0\leq r\leq s}\left|U^{(n+1)}(r,X_r^C,Y_r^C)-U^{(n)}(r,X_r^C,Y_r^C)\right|\\
\leq &\sup_{C\in \mathcal{B}[s_0,s:x_0,y_0]}\sup_{s_0\leq r\leq s}\bigg{[}\left|U^{(n)}(r,X_r^C,Y_r^C)-U^{(n-1)}(r,X_r^C,Y_r^C)\right|L(s-s_0)\\
&+\left|U^{(n+1)}(s,X_s^C,Y_s^C)-U^{(n)}(s,X_s^C,Y_s^C)\right|\bigg{]}.
\end{align*}
Note that $U^{(n)}(\cdot)$ converges on $[s,T]$, thus, for any given $\varepsilon>0$, there exists $N$ such that when $n>N$,
$$
\sup_{C\in \mathcal{B}[s_0,s:x_0,y_0]}\left|U^{(n+1)}(s,X_s^C,Y_s^C)-U^{(n)}(s,X_s^C,Y_s^C)\right|<\varepsilon,
$$
and thus
\begin{align*}
&\sup_{C\in \mathcal{B}[s_0,s:x_0,y_0]}\sup_{s_0\leq r\leq s}\left|U^{(n+1)}(r,X_r^C,Y_r^C)-U^{(n)}(r,X_r^C,Y_r^C)\right|\\
\leq &\sup_{C\in \mathcal{B}[s_0,s:x_0,y_0]}\sup_{s_0\leq r\leq s}\left|U^{(n)}(r,X_r^C,Y_r^C)-U^{(n-1)}(r,X_r^C,Y_r^C)\right|L(s-s_0)+\varepsilon.
\end{align*}

By inductive method, we complete the proof.
\end{proof}

\begin{remark}
Now, we consider how to find $t_0(M),t_1(M)$ and $C^M$ in the forward-backward system \eqref{resulty}-\eqref{verification}. However, there is no general result for the existence for this system. We try to solve it by Picard iteration. Since on $t\in[t_0(M),t_1(M)]$, $\Phi^M_t=M$. Taking derivative w.r.t $t$ on both sides yields that
\begin{equation}\label{derivative}
(r+\beta)M-\beta e^{rt}\exp\left(\int_0^t \partial_u f(r,Y^M_r,U^M_r)dr\right) \partial_y f(t,Y^M_t,U^M_t)=0.
\end{equation}

 Let $Y^{(0)}_t=y e^{-\beta t}$, which is the level of satisfaction with consumption plan $C^{(0)}\equiv 0$. Let $U^{(0)}$ be the corresponding utility function. We define
\begin{displaymath}
\hat{t}^1_0(M):=\inf\left\{\partial_y f(t, ye^{-\beta t},U^{(0)}_t)\geq \frac{M(r+\beta)}{\beta}e^{-rt}\exp\left(-\int_0^t \partial_u f(r,Y^{(0)}_r,U^{(0)}_r)dr\right)\right\}.
\end{displaymath}
Let $I^{(1)}$ be the solution of the following equation
\begin{displaymath}
(r+\beta)M-\beta e^{rt}\exp\left(\int_0^t \partial_u f(r,Y^{(0)}_r,U^{(0)}_r)dr\right) \partial_y f(t,I^{(1)}_t,U^{(0)}_t)=0.
\end{displaymath}
Then, we define
\begin{displaymath}
t_1^1(M):=\begin{cases}
\textrm{the unique solution $t$ of $\Phi^{(1)}_t=M$ in (0,T] if there is some,}\\
\textrm{$0$ otherwise},
\end{cases}
\end{displaymath}
where
\begin{displaymath}
\Phi^{(1)}_t=\beta e^{(r+\beta)t}\int_t^T \exp\left(\int_0^s \partial_u f(r,Y^{(0)}_r,U^{(0)}_r)dr\right) \partial_y f(s,I^{(1)}_t e^{-\beta(s-t)},U^{(0)}_s)e^{-\beta s}ds.
\end{displaymath}
Set $t_0^1(M)=\hat{t}_0^1(M)\wedge t_1^1(M)$ and
\begin{displaymath}
{Y}^{(1)}_t=\begin{cases}
y e^{-\beta t}, &t\in[0,t^0_0(M));\\
I^{(1)}_t, &t\in[t^1_0(M),t^1_1(M)];\\
 I^{(1)}_{t^1_1(M)} e^{-\beta(t-t_1^1(M))}, &t\in(t^1_1(M),T].
\end{cases}
\end{displaymath}
We can then obtain the associated utility function $U^{(1)}$. For any $n=2,3,\cdots$, We define
\begin{displaymath}
\hat{t}^n_0(M):=\inf\left\{\partial_y f(t, ye^{-\beta t},U^{(n-1)}_t)\geq \frac{M(r+\beta)}{\beta}e^{-rt}\exp\left(-\int_0^t \partial_u f(r,Y^{(n-1)}_r,U^{(n-1)}_r)dr\right)\right\}.
\end{displaymath}
Let $I^{(n)}$ be the solution of the following equation
\begin{displaymath}
(r+\beta)M-\beta e^{rt}\exp\left(\int_0^t \partial_u f(r,Y^{(n-1)}_r,U^{(n-1)}_r)dr\right) \partial_y f(t,I^{(n)}_t,U^{(n-1)}_t)=0.
\end{displaymath}
Then, we define
\begin{displaymath}
t_n^1(M):=\begin{cases}
\textrm{the unique solution $t$ of $\Phi^{(n)}_t=M$ in (0,T] if there is some,}\\
\textrm{$0$ otherwise},
\end{cases}
\end{displaymath}
where
\begin{displaymath}
\Phi^{(n)}_t=\beta e^{(r+\beta)t}\int_t^T \exp\left(\int_0^s \partial_u f(r,Y^{(n-1)}_r,U^{(n-1)}_r)dr\right) \partial_y f(s,I^{(n)}_t e^{-\beta(s-t)},U^{(n-1)}_s)e^{-\beta s}ds.
\end{displaymath}
Set $t_0^n(M)=\hat{t}_0^n(M)\wedge t_1^n(M)$ and
\begin{displaymath}
{Y}^{(n)}_t=\begin{cases}
y e^{-\beta t}, &t\in[0,t^0_n(M));\\
I^{(n)}_t, &t\in[t^n_0(M),t^n_1(M)];\\
 I^{(n)}_{t^n_1(M)} e^{-\beta(t-t_1^n(M))}, &t\in(t^n_1(M),T].
\end{cases}
\end{displaymath}
It remains to prove the convergence of sequence $\{U^{(n)}\}$. Note that, $t_0^n(M), t_1^n(M)$, $Y^{(n)}$ are the optimal strategies with the given utility $U^{(n-1)}$. By Theorem \ref{dpp-lem}, we have that sequence $\{U^{(n)}(\cdot)\}$ uniformly converges on $[0,T]$.
\end{remark}

\appendix
\renewcommand\thesection{Appendix}
\section{ }
\label{app}
\renewcommand\thesection{A}

In this section, we provide some technical results which are useful to obtain the property of recursive utility, the first order conditions and the construction for optimal consumption plan.

\begin{lemma}\label{simple}
	Suppose that $\delta_t=\int_t^T (g_s+h_s)ds$, $t\in[0,T]$, where $g,h$ satisfy
	\begin{itemize}
		\item[(i)] $h_s\geq -K|\delta_s|$, $s\in[0,T]$ for some constant $K>0$,
		\item[(ii)] $g_s\geq 0$, $s\in[0,T]$ and furthermore, $g_s>0$ on some interval $(t_0,t_1)$.
	\end{itemize}
	Then, we have $\delta_0>0$.
\end{lemma}

\begin{proof}
	Set $\phi_t=\int_t^T e^{-K(s-t)}g_s ds$. Then $\phi_0>0$ by the condition (ii).
		Moreover, $\phi$ satisfies the ODE
		\begin{displaymath}
			\phi'_t= K \phi - g
		\end{displaymath}
		and $\phi_T=0$.
		$\delta$ satisfies the ODE
		\begin{displaymath}
			\delta'_t= -g-h
		\end{displaymath}
		and $\delta_T=0$.		
		By the comparison theorem for ODE and (i), we have $\delta_t\geq \phi_t$, and in particular, $\delta_0 \ge \phi_0>0$.
\end{proof}


The following lemma can be regarded as a generalized Gronwall inequality.

\begin{lemma}\label{gronwall}
	Let $A=\{A_t\}$ be a bounded function, $B=\{B_t\}\in L^p([0,T], dt)$ for some $p>1$. Suppose that the right continuous function $X$   satisfies $X_T\geq C$ and
	\begin{displaymath}
		X_t\geq \int_t^s (A_rX_r+B_r)dr+X_s, \textrm{ for any } 0\leq t\leq s\leq T.
	\end{displaymath}
	Then, we have
	\begin{displaymath}
		X_t\geq \exp\left(\int_t^T A_r dr\right)C+\int_t^T\exp\left(\int_t^s A_rdr\right)B_sds.
	\end{displaymath}
\end{lemma}

\begin{remark}
	Actually, a more general result in a stochastic setting has been established in \cite{KSS17} (see Proposition B.1 in \cite{KSS17}). 
\end{remark}

Let $C^*$ be the optimal consumption plan for problem \eqref{1}. We define
\begin{displaymath}
	\phi^*_t=e^{(r+\beta)t}\int_t^T \exp\left(\int_0^s \partial_u f(r,Y_r^*,U_r^*)dr\right)\partial_y f(s,Y_s^*,U_s^*)\beta e^{-\beta s}ds,
\end{displaymath}
where $Y^*=Y^{C^*}$ and $U^*=U^{C^*}$. By Corollary \ref{cor4.5}, the support of an optimal consumption plan is contained in $\arg\max \phi^*$. The following lemma indicates that these two sets coincides and they are a closed interval in $[0,T]$ under Assumption \ref{a5.1}.

\begin{lemma}\label{lem5.6}
Under Assumption \ref{a5.1}, we have $\supp dC^*=\arg\max \phi^*$ which is a closed interval in $[0,T]$.
\end{lemma}

\begin{proof}
 First, we show that $\supp dC^*$ is a closed interval. Otherwise, we may find some $s_0<s_1\in \supp dC^*$ satisfying $(s_0,s_1)\cap \,\supp dC^*=\emptyset$. This implies that $\phi^*$ is smooth on $(s_0,s_1)$ with derivatives given by
\begin{equation}\label{derivative}\begin{split}
&\partial_t \phi^*_t=(r+\beta)\phi^*_t-\beta e^{rt} \exp(\int_0^t \partial_u f(r,Y_r^*,U_r^*)dr)\partial_y f(t,Y_t^*,U_t^*),\\
&\partial_t^2 \phi^*_t=(r+\beta)\partial_t\phi_t^*-\beta e^{rt}\exp(\int_0^t \partial_u f(r,Y_r^*,U_r^*)dr)\mathfrak{L}f(t,Y_t^*,U_t^*).
\end{split}\end{equation}
By Corollary \ref{cor4.5}, $s_0,s_1\in \arg\max \phi^*$. It follows that the minimum of $\phi^*$ over $[s_0,s_1]$ is attained at some $\hat{t}\in(s_0,s_1)$, i.e., $\hat{t}$ locally minimizes the smooth function $\phi^*$. Consequently, we have
$$\partial_t \phi^*_{\hat{t}}=0,\  \partial^2_t \phi^*_{\hat{t}}\geq 0.$$
By Eq \eqref{derivative} and noting Assumption \ref{a5.1}, we have
$$\partial_t^2 \phi^*_{\hat{t}}=-\beta e^{r{\hat{t}}}\exp\left(\int_0^{\hat{t}} \partial_u f(r,Y_r^*,U_r^*)dr\right)\mathfrak{L}f(\hat{t},Y_{\hat{t}}^*,U_{\hat{t}}^*)<0,$$
which is a contradiction.

It remains to show that $\supp dC^*=\arg\max \phi^*$. The proof is similar with the one for Lemma 5.6 in \cite{BR00}, so we omit it.
\end{proof}

\end{document}